\documentclass{hmjart}
\usepackage{amsfonts}
\usepackage{amsmath}
\usepackage[mathscr]{eucal}
\usepackage{amssymb}
\usepackage{latexsym}
\newcommand{\Hom}{\mathrm{Hom}}

\newcommand{\bN}{{\mathbb N}}
\newcommand{\bZ}{{\mathbb Z}}

\newcommand{\bR}{{\mathbb R}}
\newcommand{\bC}{{\mathbb C}}

\newcommand{\Ftwo}{{\mathbb F}_2}
\newcommand{\Fp}{{\mathbb F}_p}

\newcommand{\WAFOM}{\mathrm{WF}}
\newcommand{\Zb}{\mathbb{Z}_b}

\newcommand{\prob}{\mathrm{prob}}

\newcommand{\bx}{{\boldsymbol{x}}}

\newcommand{\bk}{{\boldsymbol{k}}}

\begin{document}
\hmjlogo{}{}{}{}                
\title{WAFOM over abelian groups for quasi-Monte Carlo point sets
}


\author[K. Suzuki]{Kosuke Suzuki}
\address{Graduate School of Mathematical Sciences\\
The University of Tokyo\\
3-8-1 Komaba, Meguro-ku, Tokyo 153-8914 JAPAN} 
\email{ksuzuki@ms.u-tokyo.ac.jp}
\grants{
The work of the author was supported by the Program for Leading Graduate Schools, MEXT, Japan.\\
The published version is available in Hiroshima Math. J. 45 (2015), no.\ 3, 341--364.
}

\begin{abstract}
In this paper, we study quasi-Monte Carlo (QMC) rules for numerical integration.
J.~Dick proved a Koksma-Hlawka type inequality
for $\alpha$-smooth integrands and gave an explicit construction of QMC rules 
achieving the optimal rate of convergence in that function class.
From this inequality,
Matsumoto, Saito and Matoba introduced
the Walsh figure of merit (WAFOM) $\WAFOM(P)$ for an $\Ftwo$-digital net $P$
as a quickly computable quality criterion for $P$ as a QMC point set.
The key ingredient for obtaining WAFOM is the Dick weight,
a generalization of the Hamming weight
and the Niederreiter-Rosenbloom-Tsfasman (NRT) weight.

We extend the notions of the Dick weight and WAFOM
over a general finite abelian group $G$, 
and show that this version of WAFOM satisfies Koksma-Hlawka type inequality
when $G$ is cyclic.
We give a MacWilliams-type identity 
on weight enumerator polynomials for the Dick weight,
by which we can compute the minimum Dick weight as well
as WAFOM\@.
We give a lower bound on WAFOM of order  $N^{-C'_G(\log N)/s}$
and an upper bound on lowest WAFOM
of order $N^{-C_G(\log N)/s}$ for given $(G,N,s)$
if $(\log N)/s$ is sufficiently large,
where $C'_G$ and $C_G$ are constants depending only on the cardinality of $G$
and $N$ is the cardinality of quadrature rules in $[0,1)^s$.
These bounds
generalize the bounds given by Yoshiki and others given for $G=\Ftwo$.
\end{abstract}

\maketitle

\section{Introduction}
Quasi-Monte Carlo (QMC) integration is a method for numerical integration
using the average of function evaluations
as an approximation of the true integration value.
In QMC integration, sample points are chosen deterministically,
while in Monte-Carlo integration they are chosen randomly.
Thus, how to construct point sets has been a major concern in QMC theory.
One of the known good construction frameworks is digital nets,
which is based on linear algebra over finite fields
(or more generally over finite rings).

A strong analogy between coding theory and QMC point sets
is well known (see, e.g., \cite{Chen2002eci,Niederreiter1992rng,Skriganov2001cta}).
In coding theory, the minimum Hamming weight is used for a criterion for linear codes.
Analogically, Niederreiter-Rosenbloom-Tsfasman (NRT) weight
is a criterion for digital nets in QMC theory  \cite{Niederreiter1986ldp,Rozenblyum1997c$m}.
More precisely, the minimum NRT weight is essentially equivalent to $t$-value
and gives an upper bound on the star-discrepancy,
which are important criteria for QMC point sets.
Recently, based on Dick's work \cite{Dick2008wsc},
Matsumoto, Saito and Matoba defined the Dick weight $\mu$ on digital nets over $\Ftwo$
and related it to a criterion called the Walsh figure of merit (WAFOM) in \cite{Matsumoto2014acf}.
In this paper, as a generalization of \cite{Matsumoto2014acf},
we extend the notions of the Dick weight and WAFOM for digital nets over $\Zb$,
and more generally, for subgroups of $G^{s \times n}$
where $G$ is a finite abelian group.
Furthermore, we establish a MacWilliams-type identity for the Dick weight,
which gives a computable formula of the minimum Dick weight and WAFOM.

Let us recall the notion of QMC integration.
For an integrable function $f \colon [0,1)^s \to \bR$
and a finite point set in an $s$-dimensional unit cube
$\mathcal{P} \subset [0,1)^s$,
quasi Monte-Carlo (QMC) integration of $f$ by $\mathcal{P}$ is an approximation value
$$
I_\mathcal{P}(f) := \frac{1}{N}\sum_{\bx \in \mathcal{P}}{f(\bx)}
$$
of the actual integration
$$
I(f) := \int_{[0,1)^s}{f(\bx) \, d\bx},
$$
where $N := |\mathcal{P}|$ is the cardinality of $\mathcal{P}$.
The QMC integration error is defined as
$
\mathrm{Err}(f;\mathcal{P}) := |I_\mathcal{P}(f) - I(f)|.
$
If the integrand $f$ has bounded variation in the sense of Hardy and Krause,
the Koksma-Hlawka inequality shows that
$
\mathrm{Err}(f;\mathcal{P}) \leq V(f)D(\mathcal{P}),
$
where $V(f)$ is the total variation of $f$
and $D(\mathcal{P})$ is the star-discrepancy of $\mathcal{P}$.
There have been many studies on the construction of low-discrepancy point sets,
i.e., point sets with $D(\mathcal{P}) \in O(N^{-1+\varepsilon})$.
In particular, digital nets and sequences are a general framework for the
construction of good point sets.
We refer to \cite{Dick2010dna} and \cite{Niederreiter1992rng} 
for the general information on QMC integration
and digital nets and sequences.

Recently, higher order convergence results for digital nets, i.e.,
$\mathrm{Err}(f;\mathcal{P})$ converges faster than $N^{-1}$,
has been established.
For a given integer $\alpha > 1$,
Dick gave quadrature rules for $\alpha$-smooth integrands
which achieve $\mathrm{Err}(f;\mathcal{P}) \in O(N^{-\alpha+\varepsilon})$ 
\cite{Dick2008wsc}.
He introduced a weight which gives a bound on a criterion
$\WAFOM_{\alpha}(\mathcal{P})$
(he did not give a name and we use the name in \cite{Matsumoto2014acf})
for a digital net $\mathcal{P}$ over a finite field with cardinality $b$,
and proved a Koksma-Hlawka type inequality
$\mathrm{Err}(f;\mathcal{P})
\leq C_{b,s,\alpha} \|f\|_\alpha \cdot \WAFOM_{\alpha}(\mathcal{P})$,
where $\|f\|_\alpha$ is a norm of $f$ for a Sobolev space and
$C_{b,s,\alpha}$ is a constant depend only on $b$, $s$, and $\alpha$.
Later he improved the constant factor of the lowest $\WAFOM_\alpha$
for digital nets over a finite cyclic group \cite{Dick2009dwc}.

As a discretized version of Dick's method,
Matsumoto, Saito and Matoba introduced
the Dick weight $\mu$ and a related criterion WAFOM
$\WAFOM(P)$ for an $\Ftwo$-digital net $P$ \cite{Matsumoto2014acf}.
WAFOM also satisfies a Koksma-Hlawka type inequality (with some errors
due to discretization).
One remarkable merit of WAFOM is that WAFOM is easily computable
by the inversion formula \cite[(4.2)]{Matsumoto2014acf},
which is easier to implement than the formula of $\mathrm{WF}_\alpha$
derived from \cite[Section~4]{Baldeaux2012ecw}.
Using this merit, they executed a random search of low-WAFOM point sets
and showed that such point sets perform better
than some standard low-discrepancy point sets.
There are several studies on low-WAFOM point sets.
The existence of low-WAFOM point sets 
was shown by Matsumoto and Yoshiki
\cite{MatsumotoYoshiki}.
The author proved that the interlacing construction for higher order QMC point sets
with Niederreiter-Xing sequences over a finite field
gives low-WAFOM point sets \cite{Suzuki2014ecp}.

In this paper, as a generalization of \cite{Matsumoto2014acf}
we propose the Dick weight and WAFOM for digital nets
over $\Zb$ and for subgroups of $G^{s \times n}$
where $G$ is a finite abelian group.
WAFOM over $\Zb$ is also a discretized version of Dick's method
and thus satisfies a Koksma-Hlawka type inequality.
Moreover, we give a MacWilliams-type identity of weight enumerator polynomials
for the Dick weight.
Using this identity we obtain a computable formula of the minimum Dick weight
as well as WAFOM,
which is a generalization of the inversion formula for WAFOM in the dyadic case.
Furthermore, we give generalizations of known properties of WAFOM over $\Ftwo$
in \cite{MatsumotoYoshiki} and \cite{Yoshiki2014lbw}.
More precisely, we give a lower bound on WAFOM and
prove the existence of low-WAFOM point sets.
In particular, we improve some of the results in \cite{MatsumotoYoshiki}.
These results imply that there exist positive constants $C, D, D'$ and  $F$
depending only on $b$ and
independent of $s$, $n$ and $N$ such that
$
N^{-C\log{N}/s}
\leq \min\{ \WAFOM(P) \mid \text{$P$ is a digital net}, |P| \leq N\}
\leq F N^{-D(\log{N})/s +D'},
$ 
if $(\log N)/s$ is sufficiently large.

These results are similar to the works of Dick,
but there is no implication between them.
Dick fixed the smoothness $\alpha$,
while our method requires $n$-smoothness on the function
where $n$ is as above.
Thus, in our case, the function class is getting smaller for $n$ being increased.

The rest of the paper is organized as follows.
In Section~\ref{sec:Preliminary}, we introduce the necessary background and notation,
such as the discretization scheme of QMC integration,
the discrete Fourier transform, and Walsh functions.
In Section~\ref{sec:WAFOM},
we define the Dick weight and WAFOM over a general finite abelian group $G$,
and prove a Koksma-Hlawka type inequality in the case that $G$ is cyclic.
In Section~\ref{sec:MacWilliams identity}, we define the weight enumerator polynomial,
give the MacWilliams-type identity for the Dick weight,
and give a computable formula of WAFOM\@.
In Section~\ref{sec:Estimation of WAFOM}, we give a lower bound on WAFOM,
prove the existence of low-WAFOM point sets,
and study the order of WAFOM.

\section{Preliminaries}\label{sec:Preliminary}
Throughout this paper, we use the following notation.
Let $\bN$ be the set of positive integers
and $\bN_0 := \bN \cup \{0\}$.
Let $b$ be an integer greater than 1.
Let $\bZ_b = \bZ/b\bZ$
be the residue class ring modulo $b$.
We identify $\bZ_b$ with the set $\{0, 1, \dots, b-1 \} \subset \bZ$.
For a set $S$, we denote by $|S|$ the cardinality of $S$.
For a group or a ring $R$ and positive integers $s$ and $n$,
we denote by $R^{s \times n}$ the set of $s \times n$ matrices
with components in $R$.
We denote by $O$ the zero matrix.
We denote by $e$ the base of the natural logarithm.

\subsection{Discretized QMC in base $b$}\label{subsec:QMC}
In this subsection, we explain discretized QMC in base $b$.
This discretization is a straightforward generalization of
the $b=2$ case in \cite{Matsumoto2014acf}.

Let $s$ be a positive integer.
Let $\mathcal{P} \subset [0,1)^s$ be a point set
in an $s$-dimensional unit cube with finite cardinality $|\mathcal{P}| = N$,
and let $f \colon [0,1)^s \to \bR$ be an integrable function.
Recall that quasi-Monte Carlo integration by $\mathcal{P}$ is an approximation value
$$
I_\mathcal{P}(f) := \frac{1}{N}\sum_{\bx \in \mathcal{P}}{f(\bx)}
$$
of the actual integration
$$
I(f) := \int_{[0,1)^s}{f(\bx) \, d\bx}.
$$
The QMC integration error is
$
\mathrm{Err}(f;\mathcal{P}) := |I_\mathcal{P}(f) - I(f)|.
$

Here, we fix a positive integer $n$,
which is called the degree of discretization or the precision.
We consider an $n$-digit discrete approximation in base $b$.
We associate a matrix $B :=(b_{i,j}) \in \bZ_b^{s \times n}$
with a point
$\bx_B = (x_B^1, \dots, x_B^s) =
(\sum_{j=1}^{n}{b_{1,j}b^{-j}}, \dots, \sum_{j=1}^{n}{b_{s,j}b^{-j}}) \in [0,1)^s$,
and with an $s$-dimensional cube 
$
\mathbf{I}_B := \prod_{i=1}^{s} I_{i}
\subset [0,1)^s
$,
where each edge
$
I_{i} := [x_B^i, x_B^i+b^{-n})
$
is a half-open interval with length $b^{-n}$.
We define $n$-digit discrete approximation $f_n$ of $f$ as
$$
f_n \colon \bZ_b^{s \times n} \to \bR,
\:\:   B :=(b_{i,j}) \mapsto \frac{1}{\mathrm{Vol}(\mathbf{I}_B)} \int_{\mathbf{I}_B}{f(\bx) \, d\bx}.
$$
Let $P$ be a subset of $\bZ_b^{s \times n}$.
We define $n$-th discretized QMC integration of $f$ by $P$ as
$$
I_{P,n}(f) := \frac{1}{|P|}\sum_{B \in P}{f_n(B)}
$$
and define the $n$-th discretized QMC integration error as
$$
\mathrm{Err}(f;P,n) := |I_{P,n}(f) - I(f)|.
$$
For each $B\in P$, we take the center point of the cube $I_B$.
Let $\mathcal{P} \subset [0,1)^s$ be the set of such center points
given by $P$.
By a slight extension of \cite[Lemma 2.1]{Matsumoto2014acf},
if $f$ is continuous with Lipschitz constant $K$ then we have
$
|I_{P,n}(f) - I_{\mathcal{P}}(f)| \leq K\sqrt{s}b^{-n}.
$
We take $n$ large enough 
so that $K\sqrt{s}b^{-n}$ is negligibly small compared to 
the order of QMC integration error $|I_{\mathcal{P}}(f) - I(f)|$
by $\mathcal{P}$.
Then we may regard the $n$-th discretized QMC integration error
$\mathrm{Err}(f;P,n)$
as an approximation of the QMC integration error $\mathrm{Err}(f;P)$.

As point sets, in this paper we consider subgroups of
$\bZ_b^{s \times n}$
as well as digital nets.
The definition of digital nets over finite rings is given in \cite{Larcher1996dna}.
we adopt an equivalent definition of digital nets,
which is proposed as digital nets with generating matrices in
\cite[Definition~4.3]{Dick2013fcw}.

\begin{definition}
Let $C_1, \dots, C_s \in \bZ_b^{n \times d}$
be matrices and let $X_1, \dots, X_d \in \bZ_b^{s \times n}$ be
defined by the $j$-th row of $X_i$ is the transpose of the $i$-th column of $C_j$.
Assume that $X_1, \dots, X_d$ are a free basis of $\bZ_b^{s \times n}$
as a $\bZ_b$-module.
For an integer $k$ with $0 \leq k \leq b^d-1$,
we define a matrix $\bx_k \in \bZ_b^{s \times n}$
as
$\bx_k = \sum_{i=1}^{d}{\kappa_{i-1} X_i}$,
where $k = \kappa_0 + \kappa_1 b^1 + \cdots + \kappa_{d-1} b^{d-1}
\,\, (0 \leq \kappa_i \leq b-1)$
is the $b$-adic expansion of $k$.
We call the set $\{\bx_0, \dots, \bx_{b^d-1} \}$ the digital net
generated by the matrices $C_1, \dots, C_s$.
\end{definition}
It is easy to see that
digital nets become subgroups of $\bZ_b^{s \times n}$.

\subsection{Discrete Fourier transform}
In this subsection, we recall the notion of character groups and
the discrete Fourier transform.
We refer to \cite{Serre1977lrf} for general information on character groups.
Let $G$ be a finite abelian group.
Let $T := \{z \in \bC \mid |z| = 1 \}$
be the multiplicative group of complex numbers of absolute value one.
Let $\omega_b = \exp(2 \pi \sqrt{-1}/ b)$.

\begin{definition}
We define the character group of $G$
by $G^\vee := \Hom(G,T)$,
namely $G^\vee$ is the set of group homomorphisms
from $G$ to $T$.
\end{definition}
There is a natural pairing
$
\bullet \colon G^\vee \times G \to T, \,\,
(h,g) \mapsto h \bullet g := h(g).
$

We can see that $\bZ_b^\vee$ is isomorphic to
$\bZ_b$ as an abstract group.
Throughout this paper, 
we identify $\bZ_b^\vee$ with $\bZ_b$ through a pairing
$
\bullet \colon \bZ_b \times \bZ_b \to T, \,\,
(h, g) \mapsto
h \bullet g := \omega_b^{hg},
$
where $hg$ is the product in $\bZ_b$.

Let $R$ be a commutative ring containing $\bC$.
Let $f:G \to R$ be a function.
We define the discrete Fourier transform of $f$ as below.

\begin{definition}
The discrete Fourier transform of $f$ is defined by
$\widehat{f} \colon G^\vee \to R, \,\,
h \mapsto \frac{1}{|G|}\sum_{g \in G}{f(g)(h \bullet g)}$.
Each value $\widehat{f}(h)$ is called a discrete Fourier coefficient.
\end{definition}

We assume that $P \subset G$ is a subgroup.
We define $P^\perp := \{h \in G^\vee \mid h \bullet g = 1 \text{ for all
} g \in P \}$.
Since
$P^\perp$ is the kernel of the restriction map $G^{\vee} \to P^{\vee}$,
we have $|P^\perp| = |G|/|P|$.
We recall the orthogonality of characters.

\begin{lemma}\label{lem:cyclicsum}
Suppose that $P \subset G$ is a subgroup and $g \in G$.
Then we have
$$\sum_{h \in P^\perp}{h \bullet g} =
\begin{cases}
|P^\perp| & \text{if $g \in P$},\\ 
0   & \text{if $g \notin P$}.
\end{cases}
$$
\end{lemma}

This lemma implies the Poisson summation formula and the Fourier inversion formula.
\begin{theorem}[Poisson summation formula]\label{thm:Poisson}
$$
\frac{1}{|P|}\sum_{g \in P}{f(g)}
= \sum_{h \in P^{\perp}}{\widehat{f}(h)}.
$$
\end{theorem}

\begin{proof}
\begin{align*}
\sum_{h \in P^{\perp}}{\widehat{f}(h)}
&= \sum_{h \in P^{\perp}} \frac{1}{|G|} {\sum_{g \in G}{f(g) (h \bullet g)}}\\
&= \sum_{g \in G} \frac{1}{|G|}f(g) \sum_{h \in P^{\perp}}  {h \bullet g}\\
&= \frac{1}{|G|} \sum_{g \in P} f(g) \cdot |P^\perp|
\quad(\because \text{Lemma}~\ref{lem:cyclicsum})\\
&= \frac{1}{|P|}\sum_{g \in P} f(g). \qed
\end{align*}
\end{proof}

\begin{theorem}[Fourier inversion formula]\label{thm:Fourier inversion}
For a complex-valued function $f \colon G \to \bC$, we have
$
f(g) = \sum_{h \in G^\vee}\widehat{f}(-h)(h \bullet g)
$
for any $g \in G$.
Moreover, if $f$ is real-valued, we have
$
f(g) = \sum_{h \in G^\vee}\overline{\widehat{f}(h)}(h \bullet g).
$
\end{theorem}
\begin{proof}
By Lemma~\ref{lem:cyclicsum}, we have $\sum_{h \in G^\vee}h \bullet g = 0$ if $g \neq 0$
and $\sum_{h \in G^\vee}h \bullet g = |G|$ if $g = 0$. Thus we have
\begin{align*}
\sum_{h \in G^\vee}\widehat{f}(-h)(h \bullet g)
&= \sum_{h \in G^\vee}\frac{1}{|G|}\sum_{g' \in G}f(g')((-h) \bullet g')(h \bullet g) \\
&= \frac{1}{|G|}\sum_{g' \in G}f(g') \sum_{h \in G^\vee}(h \bullet (g - g'))\\
&= f(g),
\end{align*}
which proves the complex-valued case.
If $f$ is real-valued, we have $\widehat{f}(-h) = \overline{\widehat{f}(h)}$,
and thus the complex-valued case implies the real-valued case.
\qed
\end{proof}

\subsection{Walsh functions}
In this subsection, we recall the notion of Walsh functions and Walsh coefficients,
and see the relationship between Walsh coefficients and discrete Fourier coefficients.
As a corollary, we prove that
the $n$-digit discrete approximation $f_n$ of $f$ is
essentially equal to the appropriate approximation of the Walsh series of $f$.
We refer to \cite[Appendix A]{Dick2010dna} for general information on Walsh functions.

First, we define Walsh functions for the one dimensional case.
\begin{definition}
Let $k \in \bN_0$ with $b$-adic expansion
$k = \kappa_0 + \kappa_1 b^1 + \kappa_2 b^2 + \cdots$
(this expansion is actually finite),
where $\kappa_j \in \{0,1, \dots, b-1\}$ for all $j \in \bN_0$.
The $k$-th $b$-adic Walsh function
${}_b \mathrm{wal}_k \colon [0,1) \to \{0, \omega_b, \dots, \omega_b^{b-1} \}$
is defined as
$$
{}_b \mathrm{wal}_k(x) := \omega_b^{\kappa_0 x_1 + \kappa_1 x_2 + \cdots},
$$
for $x \in [0,1)$ with $b$-adic expansion
$x = x_1 b^{-1} + x_2 b^{-2} + x_3 b^{-3} + \cdots$ with $x_j \in \{0,1, \dots, b-1\}$,
which is unique in the sense that infinitely many of the $x_j$ must be
different from $b-1$.
\end{definition}

This definition is generalized to the higher-dimensional case.
\begin{definition}
For dimension $s \geq 1$,
let $\bk = (k_1, \dots, k_s) \in \bN_0^s$. 
The $\bk$-th $b$-adic Walsh function
${}_b \mathrm{wal}_\bk \colon [0,1)^s \to \{0, \omega_b, \dots, \omega_b^{b-1} \}$
is defined as
$$
{}_b \mathrm{wal}_\bk (\bx) = \prod_{i=1}^{s}{{}_b \mathrm{wal}_{k_i} (x_i)}.
$$
for $\bx = (x_1, \dots, x_s) \in [0,1)^s$.
\end{definition}

Walsh coefficients are defined as follows.
\begin{definition}
Let $f \colon [0,1)^s \to \bR$. 
The $\bk$-th $b$-adic Walsh coefficient of $f$ is defined as
$$
\mathcal{F}(f)(\bk)  :=
\int_{[0,1)^s}{f(\bx) \overline{{}_b \mathrm{wal}_\bk (\bx)} \, d\bx}.
$$
\end{definition}

We see the relationship between
Walsh coefficients and discrete Fourier coefficients
in the following.
Let $A = (a_{i,j}) \in \bZ_b^{s \times n}$.
We define maps $\phi_i \colon \bZ_b^{s \times n} \to \bN_0$
as $\phi_i(A) = \sum_{j=1}^n{a_{i,j}b^{j-1}}$
and $\phi \colon \bZ_b^{s \times n} \to \bN_0^s$
as $\phi(A) = (\phi_1(A), \dots, \phi_s(A))$.
Note that $\phi_i(A) < b^n$ holds
for all $1 \leq i \leq s$ and  $A \in \bZ_b^{s \times n}$.

\begin{lemma}\label{lem:Walsh}
Let $f \colon [0,1)^s \to \bR$ and
$A = (a_{i,j}) \in \bZ_b^{s \times n}$.
Then we have
$$
\overline{\mathcal{F}(f)(\phi(A))} = \widehat{f_n}(A).
$$
\end{lemma}

\begin{proof}
Since $\phi_i(A) < b^n$ holds for all $1 \leq i \leq s$,
for all $\bx = (x_1, \dots, x_s) \in \mathbf{I}_{B}$ we have
$$
{}_b \mathrm{wal}_{\phi(A)} (\bx) = \prod_{i=1}^{s}{{}_b \mathrm{wal}_{\phi_i(A)} (x_i)}
=\prod_{i=1}^{s}{\omega_b^{a_{i,1}b_{i,1} + \dots + a_{i,n}b_{i,n}}}
= B \bullet A.
$$
Therefore we have
\begin{align*}
\overline{\mathcal{F}(f)(\phi(A))}
&= \int_{[0,1)^s}{f(\bx){}_b \mathrm{wal}_{\phi(A)} (\bx) \, d\bx}\\
&= \sum_{B \in \bZ_b^{s \times n}}
	{\int_{\mathbf{I}_{B}}{f(\bx){}_b \mathrm{wal}_{\phi(A)} (\bx) \, d\bx}}\\
&=  \sum_{B \in \bZ_b^{s \times n}}
	{\int_{\mathbf{I}_{B}}{f(\bx) (B \bullet A) \, d\bx}}\\
&= \sum_{B \in \bZ_b^{s \times n}}
	{(B \bullet A) \int_{\mathbf{I}_{B}}{f(\bx) \, d\bx}}\\
&=  \sum_{B \in \bZ_b^{s \times n}}
	{(B \bullet A) \cdot {\mathrm{Vol}(\mathbf{I}_{B})}f_n(B)}\\
&=  \sum_{B \in \bZ_b^{s \times n}}
	{(B \bullet A) \cdot b^{-sn}f_n(B)}
= \widehat{f_n}(A),
\end{align*}
which proves the lemma.\qed
\end{proof}

Let $f \sim \sum_{\bk \in \bN_0^s}\mathcal{F}(f)(\bk) {}_b\mathrm{wal}_k$
be the Walsh expansion of a real valued function $f \colon [0,1)^s \to \bR$.
Lemma~\ref{lem:Walsh} implies that
considering $n$-digit discrete approximation $f_n$ of $f$ is the same as
considering the Walsh polynomial
$\sum_{\bk < b^n}\mathcal{F}(f)(\bk) \cdot {}_b\mathrm{wal}_k$,
where $\bk = (k_1, \dots, k_s) < b^n$ means that
$k_i < b^n$ holds for every $i = 1, \dots, s$, namely we have the following.
\begin{proposition}
Let $f \colon [0,1)^s \to \bR$.
For $B \in \bZ_b^{s \times n}$, we have
$f_n(B) = \sum_{\bk < b^n}\mathcal{F}(f)(\bk) {}_b\mathrm{wal}_k (\bx_B)$.
\end{proposition}

\begin{proof}
\begin{align*}
f_n(B)
&= \sum_{A \in \bZ_b^{s \times n}}\overline{\widehat{f_n}(A)} B \bullet A \text{  ($\because$ Theorem~\ref{thm:Fourier inversion})}\\
&= \sum_{A \in \bZ_b^{s \times n}} \mathcal{F}(f)(\phi(A)) {}_b\mathrm{wal}_{\phi(A)} (\bx_B) \text{  ($\because$ Lemma~\ref{lem:Walsh})}\\
&=\sum_{\bk < b^n}\mathcal{F}(f)(\bk) {}_b\mathrm{wal}_k (\bx_B). \qed
\end{align*}
\end{proof}

\section{WAFOM over a finite abelian group}\label{sec:WAFOM}
In this section, we expand the notion of WAFOM defined in \cite{Matsumoto2014acf},
more precisely, we define WAFOM over a finite abelian group with $b$ elements.

First, we evaluate the $n$-th discretized QMC integration error of $f$
with its discrete Fourier coefficients.
Let $P \subset \bZ_b^{s \times n}$ be a subgroup.
We have
$I(f) = \widehat{f_n}(O)$
by the definition of the discrete Fourier inversion, and we have
$I_{P,n}(f) = \sum_{A \in P^\perp}{\widehat{f_n}(A)}$
by the Poisson summation formula (Theorem~\ref{thm:Poisson}).
Hence we have
$$
\mathrm{Err}(f;P,n) = |I_{P,n}(f)-I(f)| =
\left|\sum_{A \in P^\perp \backslash \{O\}}{\widehat{f_n}(A)}\right|
\leq  \sum_{A \in P^\perp \backslash \{O\}}{|\widehat{f_n}(A)|},
$$
and thus we would like to bound the value $|\widehat{f_n}(A)|$.
Dick gives an upper bound of the $\bk$-th $b$-adic Walsh coefficient $\mathcal{F}(f)(\mathbf{k})$ for
$n$-smooth function $f$ (for the definition of $n$-smoothness,
see \cite{Dick2008wsc} or \cite[\S14]{Dick2010dna}).

\begin{theorem}[\cite{Dick2010dna}, Theorem~14.23]\label{thm:Dick}
There is a constant $C_{b,s,n}$ depending only on $b, s$ and $n$ such that
for any $n$-smooth function $f \colon [0,1)^s \to \bR$ and
any $\mathbf{k} \in \bN^s$ it holds that 
$$
|\mathcal{F}(f)(\mathbf{k})| \leq C_{b,s,n} \|f\|_n \cdot b^{-\mu_n (\bk)},
$$ 
where
$\|f\|_n$ is a norm of $f$ for a Sobolev space and
$\mu_n (\bk)$ is the $n$-weight of $\bk$, which are defined in
\cite[(14.6) and Theorem~14.23]{Dick2010dna}
(we do not define them here).
\end{theorem}

Instead of $\mu_n$, we define the Dick weight $\mu$
on dual groups of general finite abelian groups below, which is a generalization of
the Dick weight over $\Ftwo$ defined in \cite{Matsumoto2014acf}.
Actually, $\mu$ is a special case of $\mu_n \circ \phi$.
More precisely, if $G = \bZ_b$ and $\alpha \geq n$ hold,
then we have $\mu = \mu_\alpha \circ \phi$ as a function from
$( \bZ_b^\vee)^{s \times n} (\simeq \bZ_b^{s \times n})$ to $\bN_0$.

\begin{definition}
Let $G$ be a finite abelian group
and let $A \in (G^\vee)^{s \times n}$.
The Dick weight $\mu \colon  (G^\vee)^{s \times n} \to \bN_0$
is defined as
$$
\mu(A):=\sum_{i,j} j\times \delta(a_{i,j}),
$$
with 
$\delta(h)=0$ for $h=0$ and 
$\delta(h)=1$ for $h\neq 0$.
\end{definition}

We obtain the next corollary.

\begin{corollary}
There exists a constant $C_{b,s,n}$ depending only on $b, s$ and $n$ such that
for any $n$-smooth function $f \colon [0,1)^s \to \bR$
and any $A \in (\bZ_b)^{s \times n}$
it holds that
$$
|\widehat{f}_n(A)| \leq C_{b,s,n} \|f\|_n \cdot b^{-\mu (A)}.
$$ 
\end{corollary}

\begin{proof}
This is the direct corollary of Theorem~\ref{thm:Dick}, Lemma~\ref{lem:Walsh},
and the equality $\mu(A) = \mu_n \circ \phi(A)$.\qed
\end{proof}

By the above corollary,
we have a bound on the $n$-th discretized QMC integration error
\begin{align*}\label{eq:generalKHineq}
\mathrm{Err}(f;P,n) := |I(f)-I_{P,n}(f)|
\leq C_{b,s,n} \|f\|_n \times \sum_{A \in P^\perp \backslash \{O\}}{b^{-\mu (A)}},
\end{align*}
for a subgroup $P$ of $\bZ_b^{s \times n}$.

Hence, as a generalization of \cite{Matsumoto2014acf}, we define a kind of figure of merit
(the Walsh figure of merit or WAFOM). 

\begin{definition}\label{def:WAFOM}
Let $s,n$ be positive integers.
Let $G$ be a finite abelian group with $b$ elements.
Let $P \subset G^{s \times n}$ be a subgroup of $ G^{s \times n}$.
We define the Walsh figure of merit of $P$ by
$$
\WAFOM(P) := \sum_{A \in P^\perp \backslash \{O\}}{b^{-\mu (A)}}.
$$
\end{definition}
In order to stress the role of the precision $n$,
we sometimes denote $\WAFOM^n(P)$ instead of $\WAFOM(P)$.

Then, as we have seen, we have the Koksma-Hlawka type inequality
$$
\mathrm{Err}(f;P,n) := |I(f)-I_{P,n}(f)|
\leq C_{b,s,n} \|f\|_n \times \WAFOM(P)
$$
for a subgroup $P \subset \bZ_b^{s \times n}$.
This shows that $\WAFOM(P)$ is a quality measure of the point set $P$
for quasi-Monte Carlo integration when $G=\bZ_b$.

\section{MacWilliams identity over an abelian group}\label{sec:MacWilliams identity}
In this section, we assume that $s, n$ are positive integers.
Recall that $G$ is a finite abelian group and $G^\vee$ its character group.
We consider an abelian group $G^{s \times n}$. 
Let $P\subset G^{s \times n}$ be a
subgroup.

We are interested in the weight enumerator polynomial of $P^\perp$
$$
W_{P^\perp}(x,y):=
\sum_{A \in P^\perp} x^{M-\mu(A)}y^{\mu(A)} \in \bC[x,y],
$$
where $M:=n(n+1)s/2$.

Let $R:=\bC[x_{i,j}(h)]$, where $x_{i,j}(h)$ is 
a family of indeterminates for $1\leq i \leq s$,
$1\leq j \leq n$, and $h \in G^\vee$. 
We define functions
$f_{i,j} \colon G^\vee \to R$ as
$f_{i,j}(h) = x_{i,j}(h)$
and
$f \colon (G^{s \times n})^\vee = (G^\vee)^{s \times n} \to R$
as
$$
f(A):=\prod_{\substack{1\leq i\leq s \\ 1\leq j\leq n}} f_{i,j}(a_{i,j})
= \prod_{\substack{1\leq i\leq s \\ 1\leq j\leq n}} x_{i,j}(a_{i,j}).
$$
Now the complete weight enumerator polynomial of $P^\perp$, in a
standard sense \cite[Chapter 5]{MacWilliams1977tec},
is defined by
$$
GW_{P^\perp}(x_{i,j}(h))
:=\sum_{A \in P^\perp} \prod_{\substack{1\leq i\leq s \\ 1\leq j\leq n}}x_{i,j}(a_{i,j}),
$$
and similarly, the complete weight enumerator polynomial of $P$
is defined by
$$
GW^*_{P}(x_{*i,j}(g))
:=\sum_{B \in P} \prod_{\substack{1\leq i\leq s \\ 1\leq j\leq n}} x_{*i,j}(b_{i,j})
$$
in $R^*:=\bC[x_{*i,j}(g)]$ where $x_{*i,j}(g)$ is 
a family of indeterminates for $1\leq i \leq s$, $1\leq j \leq n$, and $g \in G$.
We note that if we substitute
\begin{equation}\label{eq:subst}
x_{i,j}(0) \leftarrow x^j, \quad x_{i,j}(h) \leftarrow y^j
\mbox{ for }h\neq 0,
\end{equation}
we have an identity 
$$
GW_{P^\perp}(x_{i,j}(h))|_{\mbox{above substitution}}
=W_{P^\perp}(x,y).
$$

A standard formula of the Fourier transform
tells that, if 
$f_1 \colon G_1 \to R$, $f_2 \colon G_2 \to R$ are functions
and $f_1f_2 \colon G_1 \times G_2 \to R$ is their multiplication
at the value, then 
$$
\widehat{f_1f_2}=\widehat{f_1}\widehat{f_2}
$$
holds. 
This implies that 
$$
\widehat{f}(B) = \prod_{\substack{1\leq i\leq s \\ 1\leq j\leq n}}\widehat{f_{i,j}}(b_{i,j})
= \frac{1}{|G|^{sn}}\prod_{\substack{1\leq i\leq s \\ 1\leq j\leq n}}
	{\sum_{h \in G^\vee}{f_{i,j}(h)(h \bullet b_{i,j})}}.
$$
Hence, by the Poisson summation formula (Theorem~\ref{thm:Poisson}), we have
\begin{align*}
GW_{P^\perp}(x_{i,j}(h)) &= \sum_{A \in P^\perp}{f(A)}\\
&= |P^\perp| \sum_{B \in P}{\widehat{f}(B)}\\
&= \frac{1}{|P|} \prod_{\substack{1\leq i\leq s \\ 1\leq j\leq n}}
	{\sum_{h \in G^\vee}{f_{i,j}(h)(h \bullet b_{i,j})}}.
\end{align*}
Thus we have the MacWilliams identity below, which is a variant of
Generalized MacWilliams identity \cite[Chapter 5 \S6]{MacWilliams1977tec}:

\begin{proposition}[MacWilliams identity]
$$
GW_{P^\perp}(x_{i,j}(h))=
\frac{1}{|P|}
GW_{P}^*(\mbox{substituted}),
$$
where in the right hand side every $x_{*i,j}(g)$ is 
substituted by
$$
x_{*i,j}(g) \leftarrow \sum_{h \in G^\vee} (h \bullet g) x_{i,j}(h).
$$
\end{proposition}

We consider specializations of this identity.
First, we consider a specialization
$\overline{GW}_{P^\perp}(x_{1},\dots, x_{n}, y_{1}, \dots, y_{n})$
of $GW_{P^\perp}(x_{i,j}(h))$ obtained by the substitution
$$
x_{i,j}(0) \leftarrow x_j, \quad x_{i,j}(h) \leftarrow y_j
\text{ for }h\neq 0.
$$
We have
\begin{align*}
\left.\sum_{h \in G^\vee} (h \bullet g) x_{i,j}(h) \right\rvert_{\text{above substitution}}
&=(0 \bullet g)x_j + \sum_{h \in G^\vee \backslash \{0\}} (h \bullet g) y_j\\
&=x_j - y_j + \sum_{h \in G^\vee} (h \bullet g) y_j\\
&= x_j - y_j +
\begin{cases}
by_j & \text{(if $g = 0$)}\\
0    & \text{(otherwise)}
\end{cases}\\
&=
\begin{cases}
x_j + (b-1)y_j & \text{(if $g = 0$)}\\
x_j -        y_j & \text{(otherwise)}
\end{cases},\\
\end{align*}
where we use Lemma~\ref{lem:cyclicsum} for the third equality.
Thus, we have the following formula.

\begin{corollary}\label{cor:spGW}
\begin{align*}
\overline{GW}_{P^\perp}(x_1, \dots, x_n,y_1, \dots, y_n)
&= \frac{1}{|P|} \sum_{B \in P}{\prod_{\substack{1\leq i\leq s \\ 1\leq j\leq n}}
	{ (x_j + \eta(b_{i,j})y_j) } },
\end{align*}
where $\eta(b_{i,j}) = b-1$ if $b_{i,j}=0$ and $\eta(b_{i,j}) = -1$ if $b_{i,j} \neq 0$. 
\end{corollary}

Second, we consider the specialization (\ref{eq:subst}) of $GW_{P^\perp}$.
We have already seen that
$
GW_{P^\perp}\mid_{(\text{substitution (\ref{eq:subst}}))} = W_{P^\perp}(x,y)
$
holds. Since
$$
W_{P^\perp}(x,y) = \overline{GW}_{P^\perp}(x^1, \dots, x^n,y^1, \dots, y^n)
$$
follows,
Corollary~\ref{cor:spGW} implies the following formula:

\begin{theorem}\label{thm:WEP}
\begin{align*}
W_{P^\perp}(x,y)
&= \frac{1}{|P|} \sum_{B \in P}{\prod_{\substack{1\leq i\leq s \\ 1\leq j\leq n}}
	{ (x^j + \eta(b_{i,j})y^j) } },
\end{align*}
where $\eta(b_{i,j}) = b-1$ if $b_{i,j}=0$ and $\eta(b_{i,j}) = -1$ if $b_{i,j} \neq 0$. 
\end{theorem}

Using Theorem~\ref{thm:WEP}, we can compute $\WAFOM(P)$
and $\delta_{P^\perp}$, the minimum Dick weight of $P^\perp$.
The minimum Dick weight of $P^\perp$ is defined as
$$\delta_{P^\perp} := \min_{B \in P^\perp \backslash \{O\}}{\mu(B)},$$
which is used for bounding WAFOM (see Section~\ref{sec:Bounding WAFOM by the minimum weight}).
First, we introduce how to compute $\WAFOM(P)$. 
The following formula to compute WAFOM is a generalization of 
\cite[Corollary~4.2]{Matsumoto2014acf}
,which treats the case $G=\Ftwo$. 

\begin{corollary}\label{cor:computeWAFOM}
Let $P \subset \bZ_b^{s \times n}$ be a subgroup.
Then we have
$$
\WAFOM(P) = -1 + \frac{1}{|P|}\sum_{B \in P}{\prod_{\substack{1\leq i\leq s \\ 1\leq j\leq n}}
	{ (1 + \eta(b_{i,j})b^{-j})}}.
$$
\end{corollary}

\begin{proof}
\begin{align*}
\WAFOM(P) &= \sum_{A \in P^\perp \backslash \{O\}}{b^{-\mu (A)}}\\
&= -1 + \sum_{A \in P^\perp}{b^{-\mu (A)}}\\
&= -1 + W_{P^\perp}(1, b^{-1})\\
&= -1+ \frac{1}{|P|}\sum_{B \in P}{\prod_{\substack{1\leq i\leq s \\ 1\leq j\leq n}}
	{ (1 + \eta(b_{i,j})b^{-j})}}. \qed
\end{align*}
\end{proof}

The merit of Theorem~\ref{thm:WEP} and Corollary~\ref{cor:computeWAFOM}
is
that the number of summation depends only on $|P|$ linearly,
not $|P^\perp| = b^{sn}/|P|$.
We can calculate weight enumerator polynomials by
$sn$ times multiplication between an integer polynomial with a binomial,
and $|P|$ times addition of such polynomials of degree $n(n+1)/2$.
In the case of computing WAFOM,
we need $sn$ times of multiplication of real numbers and
$|P|$ times of summation of such real numbers,
thus we need $O(sn|P|)$ times of operations of real numbers.
On the other hand,
to calculate weight enumerator polynomials based on the definition,
we need ${|P^\perp|}$ times of summations of monomials,
and
to calculate weight WAFOM based on the definition,
we need ${|P^\perp|}$ times of summations of real numbers.

For QMC, the size $|P|$ cannot exceed a reasonable number of computer operations,
so $|P^\perp| = b^{sn}/|P|$ can be large if $sn$ is sufficiently large. 
This implies that the computational complexity of calculating weight enumerator polynomials or WAFOM
using Theorem~\ref{thm:WEP} or Corollary~\ref{cor:computeWAFOM}
is smaller if $sn$ is large.

Second, we introduce how to compute $\delta_{P^\perp}$.
The minimum Dick weight $\delta_{P^\perp}$
is equal to the degree of leading nonzero term of $-1+W_{P^\perp}(1,y)$,
namely:
\begin{lemma}
Let
$
W_{P^\perp}(1,y) = 1 + \sum_{i = 1}^{\infty}{a_i y^i}.
$
Then we have
$
\delta_{P^\perp} = \min\{i \mid a_i \neq 0\}.
$
\end{lemma}
Thus we can obtain the minimum Dick weight of $P^\perp$
by calculating the weight enumerator polynomial of $P^\perp$.
\begin{remark}
Because of Lemma~\ref{lem:upper bound of wt} in Section~\ref{a lower bound of WAFOM},
in order to compute $\delta_{P^\perp}$
it is sufficient to compute $W_{P^\perp}(1,y)$ only up to degree
$\delta_{P^\perp} \leq
d^2/(2s) + 3d/2 + s$.
\end{remark}

\section{Estimation of WAFOM}\label{sec:Estimation of WAFOM}
The following arguments from Section~\ref{geometry of the Dick weight}
to Section~\ref{Existence of low-WAFOM point sets}
are generalizations of \cite{MatsumotoYoshiki}
which deals with the case $G=\Ftwo$,
and arguments in Section~\ref{a lower bound of WAFOM}
are generalizations of \cite{Yoshiki2014lbw},
which deals with the case $G=\Ftwo$.
The methods for proofs are similar to \cite{MatsumotoYoshiki} and \cite{Yoshiki2014lbw}.
In this section,
we suppose that $s$ and $n$ are positive integers
and that $G$ is a finite abelian group.

\subsection{Geometry of the Dick weight}\label{geometry of the Dick weight}
Recall that $G$ is a finite abelian group with $b \geq 2$ elements,
$G^\vee$ its character group.
The Dick weight $\mu \colon (G^\vee)^{s \times n} \to \bN_0$
induces a metric
$$
d(A,B) := \mu(A-B)  \text{ for $A,B \in  (G^\vee)^{s \times n}$}
$$
and thus $(G^\vee)^{s \times n}$ can be regarded as a metric space.

Let $S_{s,n}(m) := | \{ A \in (G^\vee)^{s \times n} \mid \mu(A) = m   \} |$,
namely $S_{s,n}(m)$ is the cardinality of the sphere in $(G^\vee)^{s \times n}$
with center $0$ and radius $m$.
A combinatorial interpretation of $S_{s,n}(m)$ is as follows.
One has $s \times n$ dice. Each die has $b$ faces.
For each value $i = 1, \dots, n$, there exist exactly $s$ dice
with value 0 on one face and $i$ on the other $b-1$ faces.
Then, $S_{s,n}(m)$ is the number of ways that
the summation of the upper surfaces of $s \times n$ dice is $m$.
This combinatorial interpretation implies the following identity:
$$
\prod_{k=1}^n(1+(b-1)x^k)^s = \sum_{m=0}^\infty{S_{s,n}(m) x^m}.
$$
You can also see this identity from Theorem~\ref{thm:WEP}
for $P= \{O\}$, $x \gets 1$, and $y \gets x$.
Note that the right hand side is a finite sum.
It is easy to see that $S_{s,n}(m)$ is monotonically increasing with respect to
$s$ and $n$, and $S_{s,m}(m) = S_{s,m+1}(m) = S_{s,m+2}(m) = \cdots $ holds.

\begin{definition}
$S_s(m) := S_{s,m}(m).$
\end{definition}

We have the following identity between formal power series:
\begin{equation}\label{eq:S_s(m)}
\prod_{k=1}^\infty(1+(b-1)x^k)^s = \sum_{m=0}^\infty{S_s(m) x^m}.
\end{equation}

For any positive integer $M$, we define
$$
\mathscr{B}_{s,n}(M) :=  \{ A \in  (G^\vee)^{s \times n} \mid \mu(A) \leq M  \}, \quad 
\mathrm{vol}_{s,n}(M) := |\mathscr{B}_{s,n}(M)|,
$$
namely $\mathscr{B}_{s,n}(M)$ is the ball in $G^{s \times n}$
with center 0 and radius $M$,
and $\mathrm{vol}_{s,n}(M)$ is its cardinality.
We have $\mathrm{vol}_{s,n}(M) = \sum_{m=0}^{M}{S_{s,n}(m)}$,
and thus
$\mathrm{vol}_{s,n}(M)$ inherits properties of $S_{s,n}(m)$, namely,
$\mathrm{vol}_{s,n}(M)$ is also monotonically increasing with respect to $s$ and $n$,
and $\mathrm{vol}_{s,M}(M) = \mathrm{vol}_{s,M+1}(M) = \mathrm{vol}_{s,M+2}(M) = \dots$ holds.

\begin{definition}
$\mathrm{vol}_{s}(M) :=\mathrm{vol}_{s,M}(M)$.
\end{definition}

\subsection{Combinatorial inequalities}

\begin{lemma}\label{lem:vol}
$$
\mathrm{vol}_{s,n}(M) \leq \mathrm{vol}_s(M) \leq \exp(2\sqrt{(b-1)sM}).
$$
\end{lemma}

\begin{proof}
We have already seen the first inequality.
We prove the next inequality along \cite[Exercise~3(b), p.332]{Matouvsek1998itd},
which treats only $S=1$ and $b=2$ case.
If $M=0$ it is trivial, and so we assume that $M>0$. 
Define a polynomial with non-negative integer coefficients by
$$
f_{s,M}(x) := \prod_{k=1}^M{(1+(b-1)x^k)^s}.
$$
Since $f_{s,M}(x)$ has only non-negative coefficients, from Identity (\ref{eq:S_s(m)}) we have
$\sum_{m=0}^M S_{s}(m)x^m \leq f_{s,M}(x) \quad (x \in (0,1))$.
Hence we have
$$
\mathrm{vol}_{s}(M) = \sum_{m=0}^M S_{s}(m) \leq \sum_{m=0}^M S_{s}(M)x^{m-M} \leq f_{s,M}(x)/x^M \quad (x \in (0,1)).
$$
By taking the logarithm of the both sides and using the well-known inequality $\log(1+X) \leq X$,
for all $x \in (0,1)$ we have
\begin{align*}
\mathrm{vol}_{s,n}(M) &\leq s\sum_{k=1}^M{\log(1+ (b-1)x^k)} + M \log(1/x) \\
& < s(b-1)\sum_{k=1}^M{x^k} + M \log \left(1+ \frac{1-x}{x}\right)\\
& < s(b-1)\frac{x}{1-x} + M \frac{1-x}{x}.\\
\end{align*}
By comparison of the arithmetic mean and the geometric mean,
the last expression attains the minimum value $2\sqrt{(b-1)sM}$
when $ s(b-1)x/(1-x) = M(1-x)/x$ holds,
namely $x = (1+\sqrt{(b-1)s/M})^{-1} \in (0,1)$.\qed
\end{proof}

\begin{lemma}\label{lem:S_{s,n}}
$$
S_{s,n}(M) \leq S_s(M) \leq \exp(2\sqrt{(b-1)sM}).
$$
\end{lemma}

\begin{proof}
It follows from Lemma~\ref{lem:vol} and the inequality $S_s(M) \leq \mathrm{vol}_s(M)$.\qed
\end{proof}

\subsection{Bounding WAFOM by the minimum weight}\label{sec:Bounding WAFOM by the minimum weight}

\begin{definition}\label{def:min Dick wt}
Let $P \subset G^{s \times n}$ be a subgroup.
The minimum Dick weight of $P^\perp$ is defined by
$$
\delta_{P^\perp} := \min_{B \in P^\perp \backslash \{O\}}{\mu(B)}
$$
\end{definition}

The next lemma bounds $\WAFOM(P)$ by the minimum weight of $P^\perp$.
\begin{lemma}\label{lem:WAFOM and min Dick wt}
For a positive integer $M$, define
$$
C_{s,n}(M) := \sum_{A \in (G^\vee)^{s \times n} \backslash \mathscr{B}_{s,n}(M-1)}
{b^{-\mu(A)}}
= \sum_{m=M}^{\infty}{S_{s,n}(m)b^{-m}}
$$
and
$$
C_{s}(M) := \sum_{m=M}^{\infty}{S_{s}(m)b^{-m}}.
$$
Then we have
$$
\WAFOM^n(P) = \sum_{A \in P^\perp \backslash \{ O\}}{b^{-\mu(A)}}
\leq C_{s,n}(\delta_{P^\perp})
\leq C_{s}(\delta_{P^\perp}).
$$
\end{lemma}

\begin{proof}
The last inequality is trivial, so it suffices to prove the first inequality.
Since
$
P^\perp \backslash \{ O\}  \subset
 (G^\vee)^{s \times n} \backslash \mathscr{B}_{s,n}(\delta_{P^\perp}-1)
$
holds, we have
\begin{align*}
\WAFOM^n(P) = \sum_{A \in P^\perp \backslash \{ O\}}{b^{-\mu(A)}}
&\leq \sum_{A \in  (G^\vee)^{s \times n} \backslash \mathscr{B}_{s,n}(\delta_{P^\perp}-1)}{b^{-\mu(A)}}\\
&= C_{s,n}(\delta_{P^\perp}). \qed
\end{align*}
\end{proof}

We shall estimate $C_{s}(\lceil M' \rceil)$ ($C$ for the Complement of the ball)
for rather general real number $M'$:
from Lemma~\ref{lem:S_{s,n}} it follows that
\begin{align}\label{eq:upper bound of C}
C_{s}(\lceil M' \rceil) &= \sum_{m=\lceil M' \rceil}^{\infty}{S_{s}(m)b^{-m}}\notag\\
&\leq \sum_{m=\lceil M' \rceil}^{\infty}{b^{-m}e^{2\sqrt{(b-1)sm}}}\notag\\
&=b^{-\lceil M' \rceil} e^{2\sqrt{(b-1)s\lceil M' \rceil}}
+ \sum_{m=\lceil M' \rceil + 1}^{\infty}{b^{-m}e^{2\sqrt{(b-1)sm}}}.
\end{align}

First, we estimate the second term of the above.
The function
$$\exp(2\sqrt{(b-1)sm}) b^{-m} = \exp(2\sqrt{(b-1)sm} - m\log{b})$$
is monotonically decreasing with respect to $m$ if
\begin{align*}
\frac{d}{dm} \left(2\sqrt{(b-1)sm} - m\log{b} \right) \leq 0
&\iff \frac{2(b-1)s}{2\sqrt{(b-1)sm}} - \log{b}  \leq 0\\
&\iff \sqrt{\frac{(b-1)s}{m}} \leq \log{b}\\
&\iff m \geq (\log{b})^{-2}(b-1)s,
\end{align*}
hence we assume that $M' \geq (\log{b})^{-2}(b-1)s$.
Then, we have
\begin{align*}
&\sum_{m=\lceil M' \rceil + 1}^{\infty}{b^{-m}e^{2\sqrt{(b-1)sm}}}\\
&\leq \int_{m= \lceil M' \rceil}^\infty{e^{-m\log{b}} e^{2\sqrt{(b-1)sm}}  \, dm}\\
&= \int_{m= \lceil M' \rceil}^\infty
{\exp\left(-(\log{b})\left(\sqrt{m}-{\frac{\sqrt{(b-1)s}}{\log{b}}}\right)^2
+ \frac{(b-1)s}{\log{b}}\right)  \, dm}\\
&\leq \int_{m=M'}^\infty
{\exp\left(-(\log{b})\left(\sqrt{m}-{\frac{\sqrt{(b-1)s}}{\log{b}}}\right)^2
+ \frac{(b-1)s}{\log{b}}\right)  \, dm}\\
&= \int_{x=\sqrt{M'}}^\infty
{\exp\left(-(\log{b})\left(x-{\frac{\sqrt{(b-1)s}}{\log{b}}}\right)^2
+ \frac{(b-1)s}{\log{b}}\right)2x  \, dx}.
\end{align*}
In order to bound the last integral from above,
for a positive number $c$ we assume that
$\sqrt{M'} \geq (1+c)\sqrt{(b-1)s}/\log{b}$ or equivalently
$M' \geq (1+c)^2(\log{b})^{-2}(b-1)s$.
This assumption is stronger than the previous assumption $M' \geq (\log{b})^{-2}(b-1)s$.
Then, on the domain of integration
$x \geq \sqrt{M'} \geq (1+c)\sqrt{(b-1)s}/\log{b}$,
we have $cx \leq (1+c)(x - \sqrt{(b-1)s}/\log{b})$.
Hence the estimation continues:
\begin{align*}
& \sum_{m=\lceil M' \rceil + 1}^{\infty}{b^{-m}e^{2\sqrt{(b-1)sm}}} \\
&\leq \int_{x=\sqrt{M'}}^\infty
\exp\left(-(\log{b})\left(x-{\frac{\sqrt{(b-1)s}}{\log{b}}}\right)^2+ \frac{(b-1)s}{\log{b}}\right) \\
&\qquad\qquad\qquad\qquad \times 2\frac{1+c}{c}\left(x-\frac{\sqrt{(b-1)s}}{\log{b}}\right)  \, dx \\
&=\frac{1+c}{c}\frac{1}{\log{b}}\left[
-\exp\left(-(\log{b})\left(x-{\frac{\sqrt{(b-1)s}}{\log{b}}}\right)^2+ \frac{(b-1)s}{\log{b}}\right)
\right]_{x=\sqrt{M'}}^{\infty}\\
&= \frac{1+c}{c}\frac{1}{\log{b}}
\exp\left(
-(\log{b})\left(\sqrt{M'}-\frac{\sqrt{(b-1)s}}{\log{b}}\right)^2 + \frac{(b-1)s}{\log{b}}
\right)\\
&= \frac{1+c}{c}\frac{1}{\log{b}}
\exp(-(\log{b})M' + 2\sqrt{(b-1)sM'})\\
&=  \frac{1+c}{c}\frac{1}{\log{b}}
b^{-M'} e^{2\sqrt{(b-1)sM'}}.
\end{align*}

Second, we consider the first term of (\ref{eq:upper bound of C}).
We have already proved that
$\exp(2\sqrt{(b-1)sm}) b^{-m}$ is monotonically decreasing
if $m \geq (\log{b})^{-2}(b-1)s$,
and thus the assumption $M' \geq (\log{b})^{-2}(b-1)s$
implies
\begin{align*}
b^{-\lceil M' \rceil} e^{2\sqrt{(b-1)s\lceil M' \rceil}}
\leq b^{-M'} e^{2\sqrt{(b-1)sM'}}.
\end{align*}

Therefore we have
\begin{align*}
C_{s}(\lceil M' \rceil)
&\leq b^{-\lceil M' \rceil} e^{2\sqrt{(b-1)s\lceil M' \rceil}}
+ \sum_{m=\lceil M' \rceil + 1}^{\infty}b^{-m} {e^{2\sqrt{(b-1)sm}}}\\
&\leq b^{-M'} e^{2\sqrt{(b-1)sM'}} +
\frac{1+c}{c}\frac{1}{\log{b}} b^{-M'} e^{2\sqrt{(b-1)sM'}}\\
&= \left(1 + \frac{1+c}{c}\frac{1}{\log{b}} \right) b^{-M'} e^{2\sqrt{(b-1)sM'}}.
\end{align*}
Now we proved:

\begin{proposition}\label{prop:bound of C_{s,n}}
Let $c$ be a positive real number.
Let $M'$ be a real number with $M' \geq (1+c)^2(\log{b})^{-2}(b-1)s$.
Then we have the following bound
$$
C_{s,n}(\lceil M' \rceil)
\leq C_{s}(\lceil M' \rceil)
\leq \left(1 + \frac{1+c}{c}\frac{1}{\log{b}} \right)
b^{-M'} e^{2\sqrt{(b-1)sM'}}.
$$
\end{proposition}

\subsection{Existence of low-WAFOM point sets}\label{Existence of low-WAFOM point sets}
We denote the probability of the event $A$ by $\prob[A]$.
Let $p_b$ be the smallest prime factor of $b$.
Let $d$ be a positive integer.
Choose $d$ matrices $B_1, \dots, B_d \in G^{s \times n}$ independently and 
uniformly at random.
Let $P = \langle B_1, \dots, B_d \rangle \subset G^{s \times n}$
be the $G$-linear span
of $B_1, \dots , B_d$,
namely
$P = \{g_1B_1 + \cdots + g_dB_d \mid g_1, \dots, g_d \in G \}$
where $g \in G$ acts on $B = (b_{ij})$ by $gB = (gb_{ij})$.
Note that $|P| \leq b^d$.

\begin{remark}
If $G = \bZ_b$, by the theory of invariant factor decomposition,
we can say that there exist matrices $B'_1, \dots, B'_d$ such that
$P' := \langle B'_1, \dots, B'_d \rangle$ includes $P$ and
becomes a free $\bZ_b$-module of rank $d$.
Thus if  $G = \bZ_b$,
we can replace ``subgroup $P$'' in this subsection with a ``digital net $P$,''
since in this subsection
we consider only the existence of a subgroup which has large minimum Dick weight,
and $P \subset P'$ implies that $\delta_{P^\perp} \leq \delta_{{P'}^\perp}$.
\end{remark}

First, we evaluate $\prob[\mathrm{perp}_L]$,
where we define $\mathrm{perp}_L$ as the event that $B_1, \dots, B_d$ are all perpendicular to $L \in  (G^\vee)^{s \times n}$.

\begin{lemma}
Let $L \in (G^\vee)^{s \times n}$ be a nonzero matrix.
Then we have $\prob[L \perp B] \leq 1/p_b$.
Especially we have $\prob[\mathrm{perp}_L] \leq p_b ^{-d}$.
\end{lemma}

\begin{proof}
We consider the map
$(L\bullet) \colon G^{s \times n} \to \bC, B \mapsto L \bullet B$.
Then we have the surjective group homomorphism
$G^{s \times n} \to \mathrm{Im}(L\bullet)$,
and thus $|\mathrm{Im}(L\bullet)|$ divides $G^{s \times n}$.
Moreover, since $L$ is nonzero, $|\mathrm{Im}(L\bullet)|$ is larger than 1.
Hence we have $|\mathrm{Im}(L\bullet)| \geq p_b$.
Therefore we have
$\prob[L \perp B] = |\mathrm{Im}(L\bullet)|^{-1} \leq 1/p_b$,
and especially we have
$\prob[\mathrm{perp}_L] = \prob[L \perp B] ^d \leq p_b ^{-d}$.\qed
\end{proof}

Let $M$ be a positive integer.
We evaluate the probability of the event that $\delta_{P^\perp} \geq M$.
We have
\begin{align*}
\prob[\delta_{P^\perp} \geq M] &= 1 - \prob[\delta_{P^\perp} \leq M-1]\\
&=1 - \prob[\exists L \in \mathscr{B}_{s,n}(M-1)\backslash\{O\}
\text{ s.t. }  L \in P^\perp]\\
&=1 - \prob[\exists L \in \mathscr{B}_{s,n}(M-1)\backslash\{O\}
\text{ s.t. }  L \perp B_1, \dots, L \perp B_d]\\
&=1 - \prob[\cup_{ L \in \mathscr{B}_{s,n}(M-1)\backslash\{O\}}{\mathrm{perp}_L}]\\
&\geq 1 - \sum_{L \in \mathscr{B}_{s,n}(M-1)\backslash\{O\}}{\prob[\mathrm{perp}_L]}\\
&\geq 1 - (\mathrm{vol}_{s,n}(M-1) -1) \cdot {p_b}^{-d}\\
&> 1 - \mathrm{vol}_{s,n}(M-1) \cdot {p_b}^{-d}.
\end{align*}
This shows:

\begin{proposition}
If $\mathrm{vol}_{s,n}(M-1) \leq {p_b}^d$ holds,
then there exists a subgroup $P \subset G^{s \times n}$ with $|P| \leq b^d$
satisfying $\delta_{P^\perp} \geq M$.
\end{proposition}
By Lemma~\ref{lem:vol}, the condition of this proposition is satisfied if it holds that
\begin{equation}\label{eq:condition of M}
e^{2\sqrt{(b-1)s(M-1)}} \leq {p_b}^d \iff M \leq \frac{(\log{p_b})^2 d^2}{4(b-1)s} + 1.
\end{equation}
Therefore we have the following sufficient condition on the existence of $M$.
\begin{proposition}\label{prop:M exists?}
If $M \leq (\log{p_b})^2 d^2/(4(b-1)s) + 1$ holds,
then Inequality (\ref{eq:condition of M}) is satisfied, 
and hence there exists a subgroup $P \subset G^{s \times n}$
with $|P| \leq b^d$ satisfying $\delta_{P^\perp} \geq M$.
\end{proposition}

From now on, we define $\alpha_{b} := (\log{p_b})/2$ and $M' := A^2 d^2/((b-1)s)$
where $A \leq \alpha_{b}$ and
we take $M$ to be $\lfloor M' + 1 \rfloor$ so that
$P$ with $|P| \leq b^d$ and $\delta_{P^\perp} \geq M$ exists.
Then, by Proposition~\ref{prop:bound of C_{s,n}},
we have the following upper bound of $\WAFOM(P)$:

\begin{proposition}\label{prop:existWAFOM}
Let $\alpha_{b} := (\log{p_b})/2$.
Take a real number $A$ with $A \leq \alpha_{b}$ and
an arbitrary real number $c > 0$.
Then for any positive integers $s$, $n$, and $d \geq (1+c)(b-1)s/(A \log{b})$,
there exists a subgroup $P \subset G^{s \times n}$
with $|P| \leq b^d$ satisfying
$$
\WAFOM^n(P) \leq
	\left(1 + \frac{1+c}{c}\frac{1}{\log{b}}\right)
	b^{- A^2 d^2/((b-1)s)}e^{2A d}.
$$
\end{proposition}

\begin{proof}
Define $M' := A^2 d^2/((b-1)s)$ and
$M := \lfloor M' + 1 \rfloor$. 
By Proposition~\ref{prop:M exists?},
there exists a subgroup $P \subset G^{s \times n}$
with $|P| \leq b^d$ and $\delta_{P^\perp} \geq M$.
For this $P$, from Lemma~\ref{lem:WAFOM and min Dick wt} and Proposition~\ref{prop:bound of C_{s,n}}
we have
\begin{align*}
\WAFOM(P) &\leq C_{s}(M)\\
&= C_{s}(\lceil M'\rceil)\\
&\leq
\left(1 + \frac{1+c}{c} \frac{1}{\log{b}} \right) b^{-M'} e^{2\sqrt{(b-1)sM'}}\\
&= 	\left(1 + \frac{1+c}{c}\frac{1}{\log{b}}\right)
	b^{-A^2 d^2/((b-1)s)}e^{2A d},
\end{align*}
which proves the proposition.\qed
\end{proof}

In particular, take $A = \alpha_{b}$
and we have the next theorem.

\begin{theorem}\label{thm:existWAFOM}
Let $\alpha_{b} := (\log{p_b})/2$ and take an arbitrary real number $c > 0$.
Then for any $s$, $n$, and $d \geq (1+c)(b-1)s/(\alpha_{b} \log{b})$,
there exists a subgroup $P \subset G^{s \times n}$
with $|P| \leq b^d$ satisfying
$$
\WAFOM(P) \leq
	\left(1 + \frac{1+c}{c}\frac{1}{\log{b}}\right)
	b^{-\alpha_{b}^2 d^2/((b-1)s)}e^{2\alpha_{b} d}.
$$
\end{theorem}

Applying Theorem~\ref{thm:existWAFOM} to the case $G = \Ftwo$,
we can improve \cite[Theorem~2 and Remark~5]{MatsumotoYoshiki}.

\begin{corollary}\label{cor:WAFOMexistsF2}
Let $\alpha := \alpha_{2} = (\log{2})/2$
and take an arbitrary real number $c > 0$.
Then for any $n$ and $d \geq (1+c)s/(\alpha \log{2})$,
there exists a linear subspace $P \subset \Ftwo^{s \times n}$
with $\dim{P} \leq d$ satisfying
$$
\WAFOM(P) \leq
	\left(1 + \frac{1+c}{c}\frac{1}{\log{2}}\right)
	2^{-\alpha^2 d^2/s}e^{2\alpha d}.
$$
\end{corollary}

\begin{remark}
Suzuki \cite{Suzuki2014ecp} proved that the construction of higher order digital nets on $\Fp$
given in \cite{Dick2008wsc} combined with some Niederreiter-Xing point sets 
\cite{Niederreiter2001rpo} yields an explicit construction of low-WAFOM point sets,
whose order of WAFOM is almost the same with the order obtained in this paper.
\end{remark}

\subsection{A lower bound of WAFOM}\label{a lower bound of WAFOM}
In this subsection, we show a lower bound on WAFOM($P$),
as a generalization of \cite{Yoshiki2014lbw}.
The next lemma gives an upper bound on the minimum Dick weight
of $P^\perp$ for given $P \subset G^{s \times n}$,
which implies a lower bound of WAFOM($P$).

\begin{lemma}\label{lem:upper bound of wt}
Suppose that $s$ and $n$ are positive integers.
Let $P \subset G^{s \times n}$ be a subgroup
with $|P| \leq b^d$.
Let $q,r$ be nonnegative integers which satisfy
$d = qs+r$ and $0 \leq r < s$.
Then we have the following:
\begin{enumerate}
\item
$\delta_{P^\perp} \leq sq(q+1)/2 + (q+1)(r+1) \leq d^2/2s + 3d/2 + s$.
\item Let $C$ be an arbitrary positive real number greater than $1/2$.
If $d/s \geq (\sqrt{C + 1/16} + 3/4)/(C - 1/2)$ holds,
then we have
$\delta_{P^\perp} \leq C d^2/s$.
\end{enumerate}
\end{lemma}

\begin{proof}
We define a subgroup $Q := \{A=(a_{ij}) \in (G^{\vee})^{s \times n} \mid  a_{ij} = 0
\text{ if } (q+2 \leq j \leq n) \text{ or } (j = q+1 \text{ and } r+2 \leq i \leq s)  \}$.
We have $|Q| = b^{qs + r+1} = b^{d+1}$.
There is a $\bZ$-module isomorphism
$P^\perp / (P^\perp \cap Q)  \simeq (P^\perp + Q) / Q$,
and thus we have
$$
|P^\perp \cap Q| = \frac{|P^\perp| \cdot |Q|}{|P^\perp + Q|}
\geq \frac{b^{sn-d} \cdot b^{d+1}} {|(G^{\vee})^{s \times n}|}
= b,
$$
especially there exists a non-zero matrix $A' \in (P^\perp \cap Q)$.
Therefore we have
\begin{align*}
\delta_{P^\perp} \leq \mu(A') \leq \max \{\mu(A) \mid A=(a_{ij}) \in Q\}
= sq(q+1)/2 + (q+1)(r+1),
\end{align*}
where the last equality holds if the components of $A$ is as follows:
$$
\begin{cases}
a_{ij} = 0     \text{ if } (q+2 \leq j \leq n) \text{ or } (j = q+1 \text{ and } r+2 \leq i \leq s)  \\
a_{ij} \neq 0  \text{ if } (1 \leq j \leq q) \text{ or } (j = q+1 \text{ and } 1 \leq i \leq r+1)    \\
\end{cases}
.$$
In particular, since $q \leq d/s$ and $r+1 \leq s$, we have
\begin{align*}
\delta_{P^\perp} &\leq sq(q+1)/2 + (q+1)(r+1)\\
&\leq \frac{d}{2}\left(\frac{d}{s} + 1\right) + \left(\frac{d}{s} + 1\right)s
=\frac{d^2}{s}\left(\frac{1}{2}+\frac{3s}{2d}+\frac{s^2}{d^2}\right),
\end{align*}
which proves the first statement.

Let $C$ be a real number greater than $1/2$ and
we assume $d/s \geq (\sqrt{C + 1/16} + 3/4)/(C - 1/2)$.
Then we have $1/2 + 3s/2d + s^2/d^2 \leq C$.
Thus we obtain
$$
\delta_{P^\perp} \leq \frac{d^2}{s}\left(\frac{1}{2}+\frac{3s}{2d}+\frac{s^2}{d^2}\right)
\leq C d^2/s,
$$
which proves the second statement.\qed
\end{proof}

The above lemma gives a lower bound of $\WAFOM(P)$.

\begin{theorem}\label{thm:lower bound of WAFOM}
Suppose that $s$ and $n$ are positive integers.
Let $G$ be a finite abelian group with $b \geq 2$ elements.
Let $P \subset G^{s \times n}$ be a subgroup
with $|P| \leq b^d$.
Let $C$ be an arbitrary positive real number greater than $1/2$.
If $d/s \geq (\sqrt{C + 1/16} + 3/4)/(C - 1/2)$ holds,
then we have
$$
\WAFOM^n(P) \geq b^{-C d^2/s}.
$$
\end{theorem}

\begin{proof}
$$
\WAFOM^n(P) = \sum_{A \in P^\perp \backslash \{O\}}{b^{-\mu (A)}}
\geq {b^{-\delta_{P^\perp}}}
\geq b^{-C d^2/s}.\qed
$$
\end{proof}

\subsection{Order of WAFOM}
In this subsection,
we consider the order of
$\WAFOM(P)$
where $P$ is a subgroup of $G^{s \times n}$ with $|P| = b^d$.

We fix the base $b$.
Let $D := \alpha_{b} =
(\log{p_b}$)/2.
We fix a positive integer $E$
satisfying $E > (b-1)/(D\log{b})$.
Let $c$ be the real number such that
$E = (1+c)(b-1)/(D\log{b})$
(by the assumption that $E > (b-1)/(D\log{b})$, $c$ is positive).
Note that $c$, $D$ and $E$ depend only on $b$.

We assume that $d/s \geq E$.
Then, by Proposition~\ref{prop:existWAFOM},
there exists a subgroup $P \subset G^{s \times n}$ with $|P| \leq b^d$
satisfying
$$
\WAFOM^n(P) \leq
	\left(1 + \frac{1+c}{c}\frac{1}{\log{b}}\right)
	b^{-D^2 d^2/((b-1)s)}e^{2D d}.
$$
Moreover, by Theorem~\ref{thm:lower bound of WAFOM},
for every $P$ with $|P| \leq b^d$
we have  $\WAFOM^n(P) \geq b^{-Cd^2/s}$
where $C = (1/2 + 3/(2E) + 1/E^2)$.
Thus we have the following lemma.

\begin{lemma}\label{lem:order of WAFOM}
If $d/s \geq E$, we have
\begin{align*}
-Cd^2/s 
&\leq
\min\{\log_b(\WAFOM^n(P)) \mid P \subset G^{s \times n} \text{ subgroup}, |P| \leq b^d\} \\
&\leq -D^2 d^2/((b-1)s) + 2Dd/\log{b} +
\log_b\left(1 + \frac{1+c}{c}\frac{1}{\log{b}}\right).
\end{align*}
\end{lemma}
Especially, let $N = b^d$ and we have the following.
\begin{theorem}
Let $G$ be a finite abelian group with $|G| = b$.
Let $P \subset G^{s \times n}$ be a subgroup with $|P| \leq N$.
Let $c$, $C$, $D$, and $E$ are constants as Lemma~\ref{lem:order of WAFOM},
which depend only on $b$.
Suppose that $(\log{N})/s \geq E$.
Then we have
\begin{align*}
N^{-C(\log{N})/s}
&\leq
\min\{\WAFOM^n(P) \mid P \subset G^{s \times n} \text{ subgroup}, |P| \leq N\} \\
&\leq \left(1 + \frac{1+c}{c}\frac{1}{\log{b}}\right)
	N^{-D^2 (\log{N})/((\log{b})(b-1)s) + 2D/\log{b}}.
\end{align*}
\end{theorem}

\begin{acknowledgement}
The author would like to express his gratitude to 
his supervisor Professor Makoto Matsumoto
for the patient guidance, encouragement and many helpful discussions and comments.
\end{acknowledgement}


\begin{thebibliography}{10}

\bibitem{Baldeaux2012ecw}
Jan Baldeaux, Josef Dick, Gunther Leobacher, Dirk Nuyens, and Friedrich
  Pillichshammer.
\newblock Efficient calculation of the worst-case error and (fast)
  component-by-component construction of higher order polynomial lattice rules.
\newblock {\em Numer. Algorithms}, 59(3):403--431, 2012.

\bibitem{Chen2002eci}
W.~W.~L. Chen and M.~M. Skriganov.
\newblock Explicit constructions in the classical mean squares problem in
  irregularities of point distribution.
\newblock {\em J. Reine Angew. Math.}, 545:67--95, 2002.
\newblock translation in St. Petersburg Math. J. 13(2002), no. 2, 301-337.

\bibitem{Dick2008wsc}
Josef Dick.
\newblock Walsh spaces containing smooth functions and quasi-{M}onte {C}arlo
  rules of arbitrary high order.
\newblock {\em SIAM J. Numer. Anal.}, 46(3):1519--1553, 2008.

\bibitem{Dick2009dwc}
Josef Dick.
\newblock The decay of the {W}alsh coefficients of smooth functions.
\newblock {\em Bull. Aust. Math. Soc.}, 80(3):430--453, 2009.

\bibitem{Dick2013fcw}
Josef Dick and Makoto Matsumoto.
\newblock On the {F}ast {C}omputation of the {W}eight {E}numerator {P}olynomial
  and the {$t$} {V}alue of {D}igital {N}ets over {F}inite {A}belian {G}roups.
\newblock {\em SIAM J. Discrete Math.}, 27(3):1335--1359, 2013.

\bibitem{Dick2010dna}
Josef Dick and Friedrich Pillichshammer.
\newblock {\em Digital nets and sequences: Discrepancy theory and quasi-Monte
  Carlo integration}.
\newblock Cambridge University Press, Cambridge, 2010.

\bibitem{Larcher1996dna}
Gerhard Larcher, Harald Niederreiter, and Wolfgang~Ch. Schmid.
\newblock Digital nets and sequences constructed over finite rings and their
  application to quasi-{M}onte {C}arlo integration.
\newblock {\em Monatsh. Math.}, 121(3):231--253, 1996.

\bibitem{MacWilliams1977tec}
F.~J. MacWilliams and N.~J.~A. Sloane.
\newblock {\em The theory of error-correcting codes. {I}}.
\newblock North-Holland Publishing Co., Amsterdam, 1977.
\newblock North-Holland Mathematical Library, Vol. 16.

\bibitem{Matouvsek1998itd}
Ji{\v{r}}{\'{\i}} Matou{\v{s}}ek and Jaroslav Ne{\v{s}}et{\v{r}}il.
\newblock {\em Invitation to discrete mathematics}.
\newblock The Clarendon Press Oxford University Press, New York, 1998.

\bibitem{Matsumoto2014acf}
Makoto Matsumoto, Mutsuo Saito, and Kyle Matoba.
\newblock A computable figure of merit for quasi-{M}onte {C}arlo point sets.
\newblock {\em Mathematics of Computation}, 83(287):1233--1250, 2014.

\bibitem{MatsumotoYoshiki}
Makoto Matsumoto and Takehito Yoshiki.
\newblock Existence of higher order convergent quasi-{M}onte {C}arlo rules via
  {W}alsh figure of merit.
\newblock In {\em Monte {C}arlo and quasi-{M}onte {C}arlo methods 2012}, pages
  569--579. Springer, Berlin, 2013.

\bibitem{Niederreiter1986ldp}
Harald Niederreiter.
\newblock Low-discrepancy point sets.
\newblock {\em Monatsh. Math.}, 102(2):155--167, 1986.

\bibitem{Niederreiter1992rng}
Harald Niederreiter.
\newblock {\em Random number generation and quasi-{M}onte {C}arlo methods},
  volume~63 of {\em CBMS-NSF Regional Conference Series in Applied
  Mathematics}.
\newblock Society for Industrial and Applied Mathematics (SIAM), Philadelphia,
  PA, 1992.

\bibitem{Niederreiter2001rpo}
Harald Niederreiter and Chaoping Xing.
\newblock {\em Rational points on curves over finite fields: theory and
  applications}, volume 285 of {\em London Mathematical Society Lecture Note
  Series}.
\newblock Cambridge University Press, Cambridge, 2001.

\bibitem{Rozenblyum1997c$m}
M.~Yu. Rozenblyum and M.~A. Tsfasman.
\newblock Codes for the {$m$}-metric.
\newblock {\em Problemy Peredachi Informatsii}, 33(1):55--63, 1997.

\bibitem{Serre1977lrf}
Jean-Pierre Serre.
\newblock {\em Linear representations of finite groups}.
\newblock Springer-Verlag, New York, 1977.
\newblock Translated from the second French edition by Leonard L. Scott,
  Graduate Texts in Mathematics, Vol. 42.

\bibitem{Skriganov2001cta}
M.~M. Skriganov.
\newblock Coding theory and uniform distributions.
\newblock {\em Algebra i Analiz}, 13(2):191--239, 2001.

\bibitem{Suzuki2014ecp}
Kosuke Suzuki.
\newblock An explicit construction of point sets with large minimum {D}ick
  weight.
\newblock {\em J. Complexity}, 30(3):347--354, 2014.

\bibitem{Yoshiki2014lbw}
Takehito Yoshiki.
\newblock A lower bound on {WAFOM}.
\newblock {\em Hiroshima Math. J.}, 44(3):261--266, 2014.

\end{thebibliography}

\end{document}